\documentclass[11pt,oneside]{amsart}
\usepackage{geometry}                
\usepackage{graphicx}
\usepackage[clockwise]{rotating}
\usepackage{amssymb}
\usepackage{amsaddr}
\usepackage{color}
\usepackage{epstopdf}
\usepackage{multicol}
\newtheorem{thm}{Theorem}[section]
\newtheorem{cor}[thm]{Corollary}
\newtheorem{lem}[thm]{Lemma}
\newtheorem{prop}[thm]{Proposition}

\theoremstyle{definition}
\newtheorem{setup}[thm]{Setup}

\newtheorem{defns}[thm]{Definitions}
\newtheorem{defn}[thm]{Definition}
\newtheorem{ex}[thm]{Example}

\newtheorem{rmk}[thm]{Remark}

\newcommand{\N}{\mathbb{N}}

\newcommand{\pre}{\preceq}
\newcommand{\st}{\ensuremath{\, | \,}}

\newcommand{\Sgen}{S=\langle a_1,...,a_\nu \rangle}

\DeclareMathOperator{\Ap}{Ap}

\DeclareMathOperator{\pv}{\mathrm{pv}}

\title{Position Vectors of Numerical Semigroups}
\author{Lance Bryant and James Hamblin}
\address{Shippensburg University \\1871 Old Main Dr. \\ Shippensburg, PA 17257}
\email{lebryant@ship.edu, jehamb@ship.edu}

\begin{document}

\maketitle

\begin{abstract} We provide a new way to represent numerical semigroups by showing that the position of every Ap\'ery set of a numerical semigroup $S$ in the enumeration of the elements of $S$ is unique, and that $S$ can be re-constructed from this ``position vector."   We extend the discussion to more general objects called numerical sets, and show that there is a one-to-one correspondence between $m$-tuples of positive integers and the position vectors of numerical sets closed under addition by $m+1$. We consider the problem of determining which position vectors correspond to numerical semigroups.
\end{abstract}

\section{Introduction}\label{sec: intro}
We let $\N$ and $\N_0$ denote the positive and nonnegative integers, respectively. A {\em numerical semigroup} $S$ is a subsemigroup of $\N_0$ that contains 0 and has finite complement in $\N_0$. For two elements $u$ and $u'$ in $S$, $u \pre_S u'$ if there exists an $s\in S$ such that $u+s = u'$. This defines a partial ordering on $S$. The minimal elements in $S\setminus\{0\}$ with respect to this ordering form a unique minimal set of generators for $S$, which is denoted by $\{a_1,a_2,\dots, a_\nu\}$ where $a_1<a_2<\dots<a_\nu$. The semigroup $S=\{\sum_{i=1}^{\nu} c_ia_i : c_i\ge 0\}$ is represented using the notation $\Sgen$. Since the minimal generators of $S$ are distinct modulo $a_1$, the set of minimal generators is finite. Furthermore, having finite complement in $\N_0$ is equivalent to $\gcd\{a_i : 1\le i\le \nu\} = 1$.

The number of minimal generators of a semigroup $S$ is called the {\em embedding dimension of $S$}, and is denoted by $\nu=\nu(S)$. The element $a_1$ is called the {\em multiplicity} of $S$, and is also denoted by $e_0(S)$. When $S\ne \N_0$, we always have $2\le \nu(S)\le e_0(S)$.

For $0\ne n\in S$, the {\em Ap\'ery set} of $S$ with respect to  $n$ is the set $$\Ap(S,n) =\{w\in S : w-n\not\in S\}.$$ Every numerical semigroup containing $n$ has a unique Ap\'ery set with respect to $n$ from which much can be gleaned. Indeed, in \cite{GR} the Ap\'ery set is described as the most versatile tool in numerical semigroup theory.

We take the representation by an Ap\'ery set a step further. Consider the semigroup $S=\langle 4,7,9\rangle$. We have $\Ap(S,4) = \{0,7,9,14\}$. If we enumerate the elements of $S$ so that $S=\{\lambda_0, \lambda_1,\lambda_2,\dots\}$, where $\lambda_i<\lambda_j$ whenever $i<j$, then $\Ap(S,4)=\{\lambda_0, \lambda_2,\lambda_4,\lambda_8\}$. We can say that the position of the Ap\'ery set in the enumeration is given by $(0,2,4,8)$. It will be convenient to remove 0 from this vector and consider the difference of the components. For example, we represent $S=\langle 4,7,9\rangle$ (as a semigroup containing $4$) with the vector $(2,2,4)$ instead of $(0,2,4,8)$. This new vector has nonnegative integer components and stores equivalent information about the semigroup. We make the following definition.

\begin{defn}
Let $S = \{\lambda_0, \lambda_1, \lambda_2,\dots\}$ be a numerical semigroup containing $n\ne0$ such that $\lambda_i<\lambda_j$ whenever $i<j$. If $\Ap(S,n) = \{\lambda_0, \lambda_{x_1}, \lambda_{x_2},\dots,\lambda_{x_{n-1}}\}$, then the $(n-1)$-tuple $(x_1,x_2-x_1, x_3-x_2,\dots,x_{n-1}-x_{n-2})$ is called the {\em position vector} of $S$ with respect to $n$, and is denoted by $\mathrm{pv}_n(S)$.
\end{defn}

We show in Corollary~\ref{cor: ideal and vector} that if $S$ and $T$ are semigroups containing $n\ne0$, then $S=T$ if and only if $\pv_n(S) = \pv_n(T)$. Thus, a position vector is a representation of the semigroup. This is a rather remarkable fact. Consider the semigroup $\langle 4,7,9\rangle$. Since the position vector is $(2,2,4)$ (which means that the elements of $\Ap(\langle 4,7,9\rangle,4)$ are in positions $0$, $2$, $4$, and $8$ in the enumeration), no other semigroup can have an Ap\'ery set with elements in those same positions. This is certainly not true for minimal generating sets: the position of the minimal generators of $\langle 4,7,9\rangle$ and $\langle 4,6,9\rangle$ are the same, namely the first, second, and fourth elements in the enumeration. Nonetheless, the position vectors are $(2,2,4)$ and $(2,2,5)$ respectively.

Not every vector of positive integers is the position vector of a numerical semigroup. In Section~\ref{sec: fc and pv} we extend our consideration to more general objects than numerical semigroups, which we call numerical sets, and show that there is a one-to-one correspondence between elements of $\N^{n-1}$ and the position vectors of numerical sets closed under addition by $n$. Although we do not stress the fact in this paper, a numerical set $I$ can always be interpreted as a relative ideal of a semigroup contained in $I$, and so, in a sense, we have not extended beyond the theory of numerical semigroups.

Among the position vectors of numerical sets, we examine the problem of determining which represent numerical semigroups in Section~\ref{pv and ns}. 


\section{Position vectors of numerical sets}\label{sec: fc and pv}

As stated in the introduction, we need to work with objects more general than numerical semigroups.

\begin{defn}
A {\em numerical set} $I$ is a subset of $\N_0$ that contains $0$ and has finite complement in $\N_0$. 
\end{defn}

\begin{rmk}
A numerical set $I$ is closed under addition if and only if it is a numerical semigroup. When $I$ is not closed under addition, it is a relative ideal of the numerical semigroup $I-I$. Thus, we can think of the numerical sets as a certain collection of relative ideals which includes numerical semigroups. See \cite{BDF} for information on relative ideals and \cite{MM,PT} for numerical sets.
\end{rmk}

We need to define the position vector of a numerical set as we did for numerical semigroups in the introduction. We begin with a preliminary definition.

\begin{defn}
For $n\ne 0$, let $\Gamma_n$ be the collection of numerical sets $I$ such that $n + i \in I$ for all $i\in I$.
\end{defn}

For $I\in \Gamma_n$, we can define the Ap\'ery set with respect to $n$ as $$\Ap(I,n) = \{w\in I : w-n \not\in I\}.$$ As with numerical semigroups, it is not difficult to see that $\Ap(I,n)$ contains exactly $n$ elements of $I$ including $0$. Now we can define the position vector of the numerical set.

\begin{defn}
Let $I = \{\lambda_0, \lambda_1, \lambda_2,\dots\}$ be in $\Gamma_n$ such that $\lambda_i<\lambda_j$ whenever $i<j$. If $\Ap(I,n) = \{\lambda_0, \lambda_{x_1}, \lambda_{x_2},\dots,\lambda_{x_{n-1}}\}$, the $(n-1)$-tuple $(x_1,x_2-x_1, x_3-x_2,\dots,x_{n-1}-x_{n-2})$ is called the {\em position vector} of $I$ with respect to $n$, and is denoted by $\mathrm{pv}_n(I)$.
\end{defn}

A numerical set has multiple position vectors, but no two can have the same length. Therefore, $f: \Gamma_n \mapsto \N^{n-1}$ such that $f(I) = \mathrm{pv}_n(I)$ is a well-defined function. The goal of this section is to prove that $f$ is a one-to-one correspondence, which is proven in Theorem~\ref{thm: ideal and vector}.

We first establish Proposition~\ref{prop: permutation 1-1 conversion}, which contains a result about permutations on the set $1, 2, \dots, m$. 

\begin{defn}\label{defn: conversion}
Let $\pi = [\pi_1\cdots\pi_{m}]$ be a permutation of the set $\{1,\dots,m\}$. Then the {\em conversion vector} of $\pi$ is $r = (r_1,\dots,r_m)$, where $r_i= | \{j : j<i$ and $\pi_j < \pi_i\}|$. 
\end{defn}

Notice that in Definition~\ref{defn: conversion}, we have $0\le r_i \le i-1$ for all $1\le i\le m$. Moreover, Proposition~\ref{prop: permutation 1-1 conversion} reveals that every vector $(r_1,\dots,r_m)$ with this restriction is the conversion vector of a unique permutation. This result is similar to a well-known result about inversion vectors of permutations, see \cite{PS,T}.

\begin{prop}\label{prop: permutation 1-1 conversion}
For a fixed integer $m\ge 1$, let $r=(r_1,r_2,\dots,r_m)$ be a vector such that $0\le r_i\le i-1$, for $1\le i\le m$. Then there is a unique permutation $\pi$ of the set $\{1,\dots,m\}$ for which $r$ is its conversion vector. 
\end{prop}

\begin{proof}
Since there are exactly $m!$ such vectors and $m!$ permutation of length $m$, it suffices to show that each vector is the conversion vector of some permutation. The uniqueness follows by counting.

We will proceed by induction on $m$. If $m=1$, then $r=(0)$ and $\pi = [1]$ has $r$ as its conversion vector. Now assume that $m\ge2$ and that there is a permutation $\sigma = [\sigma_1\sigma_2\cdots\sigma_{m-1}]$ with the conversion vector $(r_1,r_2,\dots,r_{m-1})$. For each $1\le i\le m-1$, define $$\delta_i = \begin{cases} 1 &\mbox{if } \sigma_i > r_m\\ 0 &\mbox{otherwise}.\end{cases}$$ Let $\pi = [(\sigma_1+\delta_1)(\sigma_2+\delta_2)\cdots(\sigma_{m-1}+\delta_{m-1})(r_m+1)].$ It is not difficult to see that $\pi$ is a permutation, and that $(\sigma_i+\delta_i) < (\sigma_j+\delta_j)$ if and only if $\sigma_i < \sigma_j$, for $1\le i\le m-1$. Thus the conversion vector of $\pi$ is $r$.
\end{proof}

\begin{ex}\label{ex: permutation and conversion}
Consider the permutation $[42351]$. We can directly observe that the conversion vector is $(0,0,1,3,0)$. Conversely, since the proof in Proposition~\ref{prop: permutation 1-1 conversion} is constructive, we can recursively recover the permutation as follows:
\begin{eqnarray*}
[(0+1)] &=& [1]\\
\ [(1+1)(0+1)] &=& [21]\\
\ [(2+1)(1+0)(1+1)] &=& [312]\\
\ [(3+0)(1+0)(2+0)(3+1)] &=& [3124]\\
\ [(3+1)(1+1)(2+1)(4+1)(0+1)] &=& [42351].
\end{eqnarray*}
\end{ex}

If the elements of the Ap\'ery set of a numerical set $I$ are known, the position vector can be determined without considering the enumeration of all the elements of $I$. To see this, let $\Ap(I,n) = \{\lambda_{x_0},\lambda_{x_1},\dots,\lambda_{x_{n-1}}\}$. We write $\lambda_{x_i} = nk_i + \pi_i$, where $0\le \pi_i <n$. Notice that the elements of $\Ap(I,n)$ form a complete residue system modulo $n$ and $\pi_0 = 0$. Thus, $\pi = [\pi_1\cdots\pi_{n-1}]$ is a permutation of the set $\{1,\dots,n-1\}$ with a corresponding conversion vector $r = (r_1,r_2,\dots,r_{n-1})$.

\begin{thm}\label{thm: from ap to v}
Let $I\in \Gamma_n$ and $\Ap(I,n)=\{w_0,w_1,\dots,w_{n-1}\}$, where $w_i < w_j$ whenever $i<j$. We write $w_i = nk_i + \pi_i$, with $0\le \pi_i <n$. Let $r= (r_1,\dots,r_{n-1})$ be the conversion vector of $\pi = [\pi_1\pi_2\cdots\pi_{n-1}]$. Then $\pv_n(I) = (v_1,v_2,\dots,v_{n-1})$, where $v_1 = k_1+1$ and $v_i = i(k_i - k_{i-1}) + (r_i - r_{i-1})$, for $2\le i\le n-1$.
\end{thm}

\begin{proof}
Let $I = \{\lambda_0,\lambda_1,\dots\}$ and $w_i = \lambda_{x_i}$. Since $0=w_0= \lambda_{x_0}$, we have $x_0=0$. Next we show that $x_i = ik_i - \sum_{j=0}^{i-1} k_j +r_i+1$, for $1\le i\le n-1$. To do this, we compute the number of elements in $I$ that are strictly less than $nk_i$ for $1\le i\le n-1$. The sequence $(k_0,k_1,\dots,k_{n-1})$ is non-decreasing, and we set $l$ to be the largest index such that $k_i = k_l$. Now $s\in I$ and $nk_l \le s \le nk_l +(n-1)$ if and only if $s = \lambda_{x_j} + n(k_l - k_j)$ for some $0\le j \le l$. Thus \begin{eqnarray}
|\{0,1,2,\dots,nk_i -1\} \cap I | &=& |\{0,1,2,\dots,nk_l -1\} \cap I |\nonumber\\
&=& \sum_{j=0}^{l} (k_l - k_j ) \nonumber\\
&=& \sum_{j=0}^{i-1} (k_i - k_j )\nonumber\\
&=& ik_i - \sum_{j=0}^{i-1} k_j. \nonumber
\end{eqnarray} Next, $s\in I$ and $nk_i \le s < \lambda_{x_i}$ if and only if $s = \lambda_{x_j} + n(k_i - k_j) = nk_i + \pi_j,$ for some $0\le j < i$ and $\pi_j < \pi_i$. There are $r_i+1$ such elements. Therefore, we have
\begin{eqnarray}
x_i &=& |\{\lambda_0,\lambda_1,\dots,\lambda_{x_i-1}\}|\nonumber\\
&=& ik_i - \sum_{j=0}^{i-1} k_j +r_i+1\nonumber
\end{eqnarray} Recall that, by definition, $v_i = x_i - x_{i-1}$. Thus, the result now follows.
\end{proof}

Now we show that $f: \Gamma_n \mapsto \N^{n-1}$ such that $f(I) = \pv_n(I)$ is a one-to-one correspondence by constructing the inverse map. 

\begin{setup}\label{setup: v to I}
Let $v=(v_1,v_2,\dots,v_{n-1}) \in \N^{n-1}$. We will define two recursive sequences as follows:
\begin{enumerate}
\item $t_1 = 0$ and $t_i = (v_i+t_{i-1}) \bmod i,$ for $2\le i\le n-1$.
\item $l_1 = v_1 -1$ and $l_i = l_{i-1} + \frac{v_i + t_{i-1} - t_i}{i},$ for $2\le i\le n-1$.
\end{enumerate}

\noindent Notice that $0\le t_i\le i-1$ for $1\le i\le n-1$. Thus, $(t_1,t_2,\dots,t_{n-1})$ is the conversion vector of a permutation $\sigma =[\sigma_1\sigma_2\cdots\sigma_{n-1}]$ according to Proposition~\ref{prop: permutation 1-1 conversion}. Now let $$ \mathcal A_v = \{0\} \cup \{nl_i + \sigma_i : 1 \le i\le n-1\}.$$ Since $\sigma$ is a permutation on the set $\{1,\dots,n-1\}$, $\mathcal A_v$ is a complete residue system modulo $n$ that contains $0$. Thus, $\mathcal A_v$ generates a numerical set $I(\mathcal A_v) = \{w + kn : w\in A_v, k\ge 0\}$ in $\Gamma_n$. We set $g: \N^{n-1} \mapsto \Gamma_n$ such that $g(v) = I(\mathcal A_v)$.
\end{setup}

\begin{thm}\label{thm: ideal and vector}
The functions $f : \Gamma_n \mapsto \N^{n-1}$ given by $f(I) = pv_n(I)$ and $g: \N^{n-1} \mapsto \Gamma_n$ by $g(v) = I(\mathcal A_v)$ are inverse functions. Therefore, the function $f$ is a one-to-one correspondence between numerical sets closed under addition by the element $n$ and ($n-1$)-tuples of positive integers.
\end{thm}

\begin{proof}
Let $I\in \Gamma_n$ with $\Ap(I,n) = \{0,w_1,\dots,w_{n-1}\}$, and write $w_i =nk_i + \pi_i$, where $0\le \pi_i <n$. Also let $v = \pv_n(I)$ and $\mathcal A_v =\{0\} \cup \{nl_i + \sigma_i : 0\le i\le n-1\}$ as defined in Setup~\ref{setup: v to I}. It suffices to show that $\pi_i = \sigma_i$ and $k_i=l_i$ for all $1\le i\le n-1$. 

As noted before, $[\pi_1\pi_2\cdots\pi_{n-1}]$ is a permutation, and let $r=(r_1,r_2,\dots,r_{n-1})$ be its conversion vector. By Theorem~\ref{thm: from ap to v}, $r_1 = 0$, and $ r_i = (v_i + r_{i-1}) - i(k_i-k_{i-1}). $ Since $0\le r_i \le i-1$, it follows that $r_i = (v_i + r_{i-1}) \bmod i,$ for $1\le i\le n-1$. Referring to Setup~\ref{setup: v to I}, we find that $r_i = t_i$, for $1\le i\le n-1$. Thus, $r$ is the conversion vector of both $[\pi_1\pi_2\cdots\pi_{n-1}]$ and $[\sigma_1\sigma_2\cdots\sigma_{n-1}]$, and we conclude the two permutations are equal.

Again, by Theorem~\ref{thm: from ap to v} and Setup~\ref{setup: v to I}, $k_1 = v_1 - 1 = l_1$, and for $2\le i\le n-1$,
\begin{eqnarray*}
k_i -k_{i-1} &=&  \frac{v_i+r_{i-1}- r_i}{i}\\
&=& \frac{v_i+t_{i-1}- t_i}{i}\\
&=& l_i - l_{i-1}
\end{eqnarray*} 
We conclude that $k_i = l_i$, for $1\le i\le n-1$. This shows that $g\circ f = id_{\Gamma_n}$, and similarly, we obtain that $f\circ g = id_{\N^{n-1}}$. Therefore, the function $f$ is a one-to-one correspondence.
\end{proof}

The next corollary is really a restatement of Theorem~\ref{thm: ideal and vector}.

\begin{cor}\label{cor: ideal and vector}
Every numerical set closed under addition by the element $n\in \N$ has a unique position vector of length $n-1$. In particular, no two numerical semigroups have the same position vector. Moreover, every vector of length $n-1$ with entries in $\N$ is the position vector of a numerical set closed under addition by the element $n$.
\end{cor}

The next example demonstrates what we have developed in this section.
 
\begin{ex}  Let $S = \langle 6, 16,20,21,29\rangle$. This is a numerical semigroup containing $6$, and hence $S\in \Gamma_6$. We can compute $$\Ap(S,6) = \{0, 16, 20, 21, 29, 37\} = \{0, 6(2) + 4, 6(3) + 2, 6(3)+3, 6(4)+5, 6(6)+1\}.$$ The permutation $[42351]$ has conversion vector $(0,0,1,3,0)$. Thus,
\begin{eqnarray*}
v_1 &=& 2+1 = 3\\
v_2 &=& 2(3-2) + (0-0) = 2\\
v_3 &=& 3(3-3) + (1-0) = 1\\
v_4 &=& 4(4-3) + (3-1) = 6\\
v_5 &=& 5(6-4) + (0-3) = 7.\\
\end{eqnarray*}
So, we have $\pv_6(S) = (3,2,1,6,7)$.

Conversely, suppose we start with $v=(3,2,1,6,7) \in \N^{5}$. According to Setup~\ref{setup: v to I} and Theorem~\ref{thm: ideal and vector}
, the $r_i$'s are $(0,0,1,3,0)$ and the $k_i$'s $(2,3,3,4,6)$. From the conversion vector $(0,0,1,3,0)$, we construct the corresponding permutation $[42351]$ (see Example~\ref{ex: permutation and conversion}). Now,
\begin{eqnarray*}
w_0 &=&  0\\
w_1 &=& 6(2)+ 4 = 16\\
w_2 &=& 6(3)+ 2 = 20\\
w_3 &=& 6(3)+ 3 = 21\\
w_4 &=& 6(4)+ 5 = 29\\
w_5 &=& 6(6)+ 1 = 37.\\
\end{eqnarray*}
Thus, $S = \{w_i + c_in : c_i \ge 0\} =\langle 6, 16,20,21,29\rangle$. 
\end{ex}

\section{Position vectors of numerical semigroups}\label{pv and ns}

Now that we have a one-to-one correspondence between $\N^{n-1}$ and numerical sets closed under addition by $n$ established in the previous section, we want to know which position vectors correspond to semigroups. We provide a method for solving this problem and give an explicit answer for semigroups containing small numbers. 

We begin with a necessary and sufficient condition for a complete residue system modulo $n$ that contains 0 to be the Ap\'ery set of a numerical semigroup. This result is similar to others contained in \cite{BGGR,K}, and the proof is omitted.

\begin{lem}\label{lem: crs and ns}
Let $\mathcal A =\{w_0,w_1,\dots,w_{n-1}\}$ be a complete residue system modulo $n$ that contains $0$, where $w_0<w_1<\cdots< w_{n-1}$. Then, $I(\mathcal A) = \{w + kn : w\in \mathcal A, k\ge 0\}$ is a numerical semigroup if and only if $w_i+w_j \ge w_l$ whenever $w_i+w_j \equiv w_l \bmod n$ with $0<i\le j < l$.
\end{lem}

We can now translate Lemma~\ref{lem: crs and ns} into a condition concerning the position vector, but first we need a few preliminary results.

\begin{lem}\label{lem: gamma guy}
Let $I \in \Gamma_n$ be a numerical set with Ap\'ery set $\Ap(I,n)=\{w_0,w_1,\dots,w_{n-1}\}$, where $w_0<w_i<\dots<w_{n-1}$. We set $w_i=nk_i+\pi_i$ for $0\le i\le n-1$. If $v=(v_1,\dots,v_{n-1})$ is the position vector of $I$, then $k_1 = v_1-1$ and $$k_i-k_{i-1} = \left\lfloor\frac{v_i-1}{i}\right\rfloor +\gamma_i,$$ where $\gamma_1=0$ and $$\gamma_i = \begin{cases} 
      0 & \textrm{ if $\pi_{i-1} < \pi_i$} \\
      1 & \textrm{ if $\pi_{i-1} > \pi_i$} \\
   \end{cases},$$ for $2\le i\le n-1$.
\end{lem}

\begin{proof}
It follows from Theorem~\ref{thm: from ap to v} that $k_1 = v_1-1$ and $v_i = i(k_i-k_{i-1}) +(r_i-r_{i-1}),$ for $2\le i\le n-1$, where $r=(r_1,\dots,r_{n-1})$ is the conversion vector of $\pi = [\pi_1\cdots\pi_{n-1}]$. We rewrite this as $v_i-1 = i(k_i-k_{i-1}-\gamma_i) +(i\gamma_i+r_i-r_{i-1}-1)$. If we show that $0\le i\gamma_i+r_i-r_{i-1}-1<i$ whenever $2\le i\le n-1$, then it follows that $$k_i-k_{i-1} = \left\lfloor\frac{v_i-1}{i}\right\rfloor +\gamma_i.$$ First suppose that $r_i > r_{i-1}$. By the definition of the conversion vector, we have $\pi_{i-1}< \pi_i $ and recall that $0\le r_j\le j-1$ for all $1\le j\le n-1$. Thus, $\gamma_i=0$ and $0\le r_i-r_{i-1}-1 \le i-2$. Next, suppose that $r_i \le r_{i-1}$ so that $\pi_{i-1} > \pi_i$. Then $\gamma_i=1$ and $1\le i +r_i-r_{i-1}-1\le i-1$, which completes the proof.
\end{proof}

The next proposition will lead to a convenient equivalence relation on the elements of $\N^{n-1}$, i.e., the position vectors of numerical sets in $\Gamma_n$.

\begin{prop}\label{prop: perm class}
Let $v = (v_1,\dots,v_m)$ and $z = (z_1,\dots,z_m)$ be two position vectors with associated conversion vectors and permutations denoted by $r$ and $\pi$ for $v$, and $s$ and $\sigma$ for $z$. Then the following are equivalent:
\begin{enumerate}
\item $v_i \equiv z_i \bmod i$ for all $1\le i\le m$ 
\item $r = s$
\item $\pi = \sigma$. 
\end{enumerate}
\end{prop}

\begin{proof}
For (1) implies (2), by Setup~\ref{setup: v to I}, the conversion vector depends on the position vector modulo $i$ for the $i$-th entry. The equivalence of (2) and (3) follows from Proposition~\ref{prop: permutation 1-1 conversion}. Lastly, (2) implies (1) since, according to Theorem~\ref{thm: from ap to v}, $v_i \equiv r_i -r_{i-1} \bmod i$.
\end{proof}

\begin{defn}
We say two elements $v=(v_1,\dots,v_m)$ and $z=(z_1,\dots,z_m)$ of $\N^m$ are {\em congruent}, denoted by $v \sim z$, if $v_i\equiv z_i \bmod i$ for all $1\le i\le m$. In this case, $v$ and $z$ have the same associated permutation $\pi$ and are said to be in the {\em permutation class} defined by $\pi$.
\end{defn}

We can now present the main result of this section.

\begin{thm}\label{thm: pv for ns}
Let $(v_1,v_2,\dots,v_m)$ be a vector of positive integers in the permutation class defined by $[\pi_1\pi_2\cdots\pi_m]$. Also set $\gamma_i$ as in Lemma~\ref{lem: gamma guy} and $$ u_i = \left\lfloor\frac{v_i-1}{i}\right\rfloor.$$ Then $(v_1,v_2,\dots,v_m)$ is the position vector of a numerical semigroup if and only if $$\sum_{x=1}^i (u_x+\gamma_x) + \frac{\pi_i+\pi_j-\pi_l}{m+1} \ge \sum_{x=j+1}^l (u_x+\gamma_x),$$ whenever $0<i\le j<l$ and $\pi_i+\pi_j \equiv \pi_l \bmod (m+1)$.
\end{thm}

\begin{proof}
Let $(v_1,v_2,\dots,v_m)$ be the position vector of a numerical semigroup $S$ with Ap\'ery set $\Ap(S,m+1)=\{w_0,\dots,w_{m}\}$, where $w_i = (m+1)k_i+\pi_i$ for $1\le i\le m$. If $0<i\le j<l$ and $\pi_i+\pi_j \equiv \pi_l \bmod (m+1)$, then $w_i+w_j \equiv w_l \bmod (m+1)$ and by Lemma~\ref{lem: crs and ns}, we have $w_i+w_j \ge w_l$. Thus,
\begin{eqnarray*}
w_i+w_j &\ge& w_l\\
k_i  +\frac{\pi_i +\pi_j-\pi_l}{m+1} &\ge& k_l-k_j\\
\sum_{x=1}^i (k_x-k_{x-1}) + \frac{\pi_i+\pi_j-\pi_l}{m+1} &\ge& \sum_{x=j+1}^l (k_x-k_{x-1})\\
\sum_{x=1}^i (u_x+\gamma_x) + \frac{\pi_i+\pi_j-\pi_l}{m+1} &\ge& \sum_{x=j+1}^l (u_x+\gamma_x).
\end{eqnarray*}
Essentially reversing these steps provides the converse argument and finishes the proof.
\end{proof}

The next example applies Theorem~\ref{thm: pv for ns} to numerical semigroups containing 3.

\begin{ex}\label{ex: pv for e3}
Every numerical set closed under addition by 3 belongs to one of two permutation classes, namely one defined by permutation $[12]$ or the permutation $[21]$. We consider these two cases separately:
\begin{enumerate}
\item For the permutation $[12]$, we have $\pi_1+\pi_1\equiv \pi_2\bmod 3$, $\gamma_1=0$, and $\gamma_2=0$. Thus, 
\begin{eqnarray*}
u_1 + \frac{1+1-2}{3} &\ge& u_2\\
u_1 &\ge& u_2.
\end{eqnarray*}
\item For the permutation $[21]$, we have $\pi_1+\pi_1\equiv \pi_2\bmod 3$, $\gamma_1=0$, and $\gamma_2=1$. Thus, 
\begin{eqnarray*}
u_1 + \frac{2+2-1}{3} &\ge& u_2+1\\
u_1 &\ge& u_2.
\end{eqnarray*}
\end{enumerate} 
We conclude that the vector $(v_1,v_2)$ corresponds to a numerical semigroup if and only if $u_1\ge u_2$, or equivalently, $$v_1-1 \ge \left\lfloor\frac{v_2-1}{2}\right\rfloor.$$
\end{ex}

Using the method derived from Theorem~\ref{thm: pv for ns} and demonstrated in Example~\ref{ex: pv for e3}, we summarize the computational results for semigroups containing $n$ where $2\le n\le 5$. We omit the details.

\begin{thm}\label{thm: pv for e3 e4} Let $(v_1,v_2,\dots,v_{n-1})$ be an $(n-1)$-tuple of positive integers and set $$ u_i = \left\lfloor\frac{v_i-1}{i}\right\rfloor.$$ The following is a list necessary and sufficient conditions for $v$ to be the position vector of the a numerical semigroup containing $n$, for $2\le n\le 5$.
\begin{itemize}
\item n=2: $(v_1)$ with no restriction
\item n=3: $(v_1,v_2)$ such that $u_1\ge u_2$.
\item n=4: $(v_1,v_2,v_3)$ with restrictions given in Table~1
\begin{table}[htdp]\label{tab: table1}
\begin{center}
\begin{tabular}{|c|c|} \hline
$\sim$ to one of & satisfying \\ \hline
$(1,1,1)$,  $(1,2,3)$ & $u_1\ge u_2$ and $u_1\ge u_3$ \\ \hline 
$(1,1,2)$, $(1,2,2)$ &$u_1\ge u_3$ \\ \hline
$(1,2,1)$ & $u_1\ge u_2+u_3$ \\ \hline
$(1,1,3)$ & $u_1\ge u_2+u_3+1$\\\hline
\end{tabular}
\end{center}\caption{Restrictions for a 3-tuple to represent a semigroup}
\end{table}

\item n=5: $(v_1,v_2,v_3,v_4)$ with restrictions given in Table~2
\begin{table}[htdp]
\begin{center}
\begin{tabular}{|c|c|} \hline
$\sim$ to one of & satisfying \\ \hline

$(1,1,1,1)$,  $(1,1,2,2)$,  & $u_1\ge u_2$, $u_1\ge u_3$, $u_1\ge u_4$,  \\  
$(1,2,2,3)$, $(1,2,3,4)$ & and $u_1+u_2\ge u_3+u_4$ \\ \hline

$(1,2,1,2)$, $(1,2,3,1)$ & $u_1\ge u_2$ and $u_1 \ge u_3 + u_4$ \\ \hline

$(1,1,3,3)$, $(1,1,1,4)$ & $u_1\ge u_2$ and $u_1 \ge u_3 + u_4+1$ \\ \hline

$(1,1,1,2)$, $(1,2,1,4)$ & $u_1\ge u_2+u_3$, $u_1 \ge u_4$, \\
& and $u_1+u_2\ge u_3+u_4$ \\ \hline

$(1,2,3,3)$, $(1,1,3,1)$ & $u_1\ge u_2+u_3+1$, $u_1 \ge u_4$, \\
& and $u_1+u_2\ge u_3+u_4$ \\ \hline

$(1,1,1,3)$,  $(1,2,3,2)$,  & $u_1\ge u_2+u_3+u_4+1$  \\  
$(1,2,2,1)$, $(1,1,2,4)$, & \\
 $(1,1,2,3)$, $(1,2,2,2)$ &\\ \hline
 
 $(1,1,3,4)$  & $u_1\ge u_2+u_3+u_4+2$ \\ \hline
 
  $(1,2,1,1)$  & $u_1\ge u_2+u_3+u_4$ \\ \hline
  
$(1,1,2,1)$, $(1,2,1,3)$ & $u_1\ge u_2+u_3$ and $u_1 \ge u_3 + u_4$ \\ \hline

$(1,2,2,4)$, $(1,1,3,2)$ & $u_1\ge u_2+u_3+1$ and $u_1 \ge u_3 + u_4+1$ \\ \hline

\hline
\end{tabular}
\end{center}\caption{Restrictions for a 4-tuple to represent a semigroup}
\end{table}
\end{itemize}
\end{thm}

In Theorem~\ref{thm: pv for e3 e4}, $n$ is an element of the semigroup. By adding an extra restriction, we can force $n$ to be the multiplicity.

\begin{prop}
Let $S$ be a semigroup containing $n$ with position vector $(v_1,v_2,\dots,v_{n-1})$. Then $n$ is the multiplicity of $S$ if and only if $v_1>1$.
\end{prop}

\begin{proof}
We always have $v_1\ge 1$. If $v_1> 1$, then the first nonzero element of $S$ is not in $\Ap(S,n)$. Thus the first nonzero element of $S$ cannot be smaller than $n$, and so $n$ is the multiplicity of $S$. If $v_1 = 1$, then the first nonzero element of $S$ is in $\Ap(S,n)$. Thus the first nonzero element of $S$ is smaller than $n$, and $n$ is not the multiplicity of $S$. 
\end{proof}

\begin{rmk}\label{rmk: res. to semigroups}
It follows that in Theorem~\ref{thm: pv for e3 e4}, we can add the restriction $u_1>0$ to ensure that $n$ is the multiplicity of $S$.
\end{rmk}

\end{document}